\documentclass{amsart}

\usepackage{amsfonts,amssymb,amsmath}
\usepackage{graphicx}
\usepackage{psfrag}

\newtheorem{theo}{Theorem}[section]
\newtheorem{lem}[theo]{Lemma}

\newtheorem{cor}{Corollary}[section]

\theoremstyle{definition}
\newtheorem{defi}[theo]{Definition}

\theoremstyle{remark}
\newtheorem{rem}[theo]{Remark}

\newtheorem{exam}{Example}[section]

\numberwithin{equation}{section}

\newcommand{\bR}{{\mathbb R}}
\newcommand{\bZ}{{\mathbb Z}}
\newcommand{\bZp}{{{\mathbb Z}^+}}

\newcommand{\bfi}{{\mathbf i}}
\newcommand{\bfj}{{\mathbf j}}

\newcommand{\bfx}{{\mathbf x}}

\newcommand{\SB}{{\mathcal B}}

\newcommand{\SH}{{\mathcal H}}

\def\gep{\varepsilon}
\def\ga{\alpha}
\def\gb{\beta}
\def\gL{\Lambda}

\def\wdt{\widetilde}

\def\PVn{\mathcal{P}_V^{(n)}}

\def\N{\mathcal{N}}

\begin{document}

\title[Recurrent fractal interpolation surfaces]{Construction and box dimension of \\ recurrent fractal interpolation surfaces}



\author{Zhen Liang}
\address{Department of Statistics, Nanjing Audit University, Nanjing 211815, China}
\email{zhenliang@nau.edu.cn}

\author{Huo-Jun Ruan}
\address{School of Mathematical Sciences, Zhejiang University, Hangzhou 310027, China}
\email{ruanhj@zju.edu.cn}
\thanks{The research was supported in part by
the NSFC grant 11771391, and by ZJNSFC grant LR14A010001.}



\thanks{Corresponding author: Huo-Jun Ruan}

\subjclass[2010]{Primary 28A80; Secondary 41A30}

\date{}


\keywords{Recurrent fractal interpolation surfaces, box dimension, iterated function systems, vertical scaling factors}

\begin{abstract}
  In this paper, we present a general framework to construct recurrent fractal interpolation surfaces (RFISs) on rectangular grids. Then we introduce bilinear RFISs, which are easy to be generated while there are no restrictions on interpolation points and vertical scaling factors. We also obtain the box dimension of bilinear RFISs under certain constraints, where the main assumption is that vertical scaling factors have uniform sums under a compatible partition.
\end{abstract}

\maketitle

\section{Introduction}
\label{intro}







Fractal interpolation functions (FIFs) were introduced by  Barnsley \cite{Bar86} in 1986. The graphs of these functions are invariant sets of certain iterated function systems (IFSs) and they are ideally suited for the approximation of naturally occurring functions.  In \cite{BEH89}, Barnsley, Elton and Hardin generalized the notion of FIFs to recurrent FIFs, whose graphs are invariant sets of certain recurrent IFSs. There are many theoretic works and applications of FIFs and recurrent FIFs. Please see \cite{BRSarx,Bar93, BHVV16, CKV15, MaHa92, VDL94, WY13} for examples.

Naturally, we want to generalize FIFs and recurrent FIFs to higher dimensional cases, especially, the two-dimensional case. However, while it is easy to construct similar IFSs and recurrent IFSs as in one dimensional case, it is hard to guarantee that their invariant sets are graphs of continuous functions.

In \cite{Mas90}, Massopust introduced fractal interpolation surfaces (FISs) on triangles, where the interpolation points on the boundary are required to be coplanar. Dalla~\cite{Dal02} constructed FISs on rectangular grids, where the interpolation points on the boundary are collinear. Feng \cite{Fen08} presented a more general construction of FISs on rectangular grids, while there are restrictive conditions on vertical scaling factors.

Using a `fold-out' technique introduced by Ma{\l}ysz \cite{Mal06} and developed by Metzler and Yun \cite{MeYu10}, Ruan and Xu \cite{RuanXu15} presented  a general framework to construct FISs on rectangular grids. Based on this work, Verma and Viswanathan~\cite{VerVisArx} studied a bounded linear operator analogous to the so-called $\ga$-fractal operator associated with the univariate FIFs \cite{Nav10}. Ruan and Xu \cite{RuanXu15} also introduced bilinear FISs. One advantage of bilinear FISs is that there is no restriction on interpolation points and vertical scaling factors. We remark that some ideas of \cite{RuanXu15} are similar to the paper on the one-dimensional bilinear FIFs by Barnley and Massopust \cite{BaMa15}.


As pointed in \cite{RuanXu15}, we expect that bilinear FISs will be used to generate some natural scenes. However, in order to fit given data more effectively, we need recurrent FISs (RFISs). Bouboulis, Dalla and Drakopoulos~\cite{BoDa07,BDD06}, and Yun, Choi and O~\cite{YCO15} presented some nice methods to generate RFISs with function scaling factors, while the restrictive conditions for continuity are hard to check.

In this paper, by extending the methods in \cite{BEH89} and \cite{RuanXu15}, we present a general framework to construct RFISs on rectangular grids and introduce bilinear RFISs. Similarly to bilinear FISs constructed in \cite{RuanXu15}, bilinear RFISs can be defined on rectangular grids without any restrictions on interpolation points and vertical scaling factors.

It is one of the basic problems in fractal geometry to obtain the box dimension of fractal sets. In \cite{BEH89}, Barnely, Elton and Hardin obtained the box dimension of affine recurrent FIFs in one-dimensional case, where they assumed that the correlation matrix is irreducible. The result was generalized by Ruan, Sha and Ye \cite{RSY00}, where there is no restrict condition on the correlation matrix. We remark that in \cite{BEH89,RSY00} and in the later works such as \cite{BoDa07,BDD06}, the vertical scaling factors are constants.

In the case of bilinear FIFs (in one dimensional case) and bilinear FISs, it is very difficult to obtain their box dimension without additional assumptions since vertical scaling factors are functions. The main assumption in \cite{BaMa15,KRZ18} is that vertical scaling factors have a uniform sum.

In this paper, stronger restrict conditions are needed to obtain the box dimension since we deal with bilinear RFISs. We assume that vertical scaling factors have uniform sums under a compatible partition. From Lemmas~\ref{lem:vertScal} and~\ref{lem:Osi-local-ineq-1}, we can see that the assumption is reasonable.

We remark that by using the method of this paper and other previous papers, the authors also studied the construction and the box dimension of RFISs on triangular domains in \cite{LR19}.


The paper is organized as follows. In section~2, we present a general framework to construct RFISs. Bilinear RFISs are introduced in Section~3. In Section~4 we obtain the box dimension of bilinear RFISs under certain constraints.



\section{Construction of Recurrent Fractal Interpolation Surfaces}

For every positive integer $N$, we denote $\Sigma_N=\{1,2,\ldots,N\}$ and $\Sigma_{N,0}=\{0,1,\ldots,N\}$.


Given a data set $\{(x_{i}, y_{j},z_{ij})\in \mathbb{R}^{3}:\,i\in\Sigma_{N,0},j\in\Sigma_{M,0}\}$ with
$$x_{0}<x_{1}<\cdots<x_{N}, \quad y_{0}<y_{1}<\cdots<y_{M},$$
where $N,M\geq 2$ are positive integers, we want to construct a fractal function which interpolates the data set.

Denote $I=[x_0,x_N]$ and $J=[y_0,y_M]$. For any $i\in\Sigma_{N}$ and $j\in\Sigma_{M}$, we denote
$I_{i}=[x_{i-1},x_{i}], J_{j}=[y_{j-1},y_{j}]$, and $D_{ij}=I_i\times J_j$.


We choose $x_0^\prime, x_1^\prime,\ldots,x_N^\prime$ to be a finite sequence in $\{x_i:\, i\in \Sigma_{N,0}\}$ such that
\begin{equation*}
  |x_{i}^\prime-x_{i-1}^\prime| > x_i-x_{i-1}, \quad \forall \, i\in \Sigma_N.
\end{equation*}
For each $i\in \Sigma_N$, we denote by $I'_i$ the closed interval with endpoints $x_{i-1}^\prime$ and $x_i^\prime$, and let $u_{i}:\, I_i^\prime\to I_i$ be a contractive homeomorphism with
\begin{equation}\label{eq:ui-cond}
 u_i(x_{i-1}^\prime)=x_{i-1}, \quad u_i(x_i^\prime)=x_i.
\end{equation}
Clearly, $u_i(x_i^\prime)=u_{i+1}(x_i^\prime)=x_i$ for all $i\in \Sigma_{N-1}$, which is an important property in our construction.

Similarly, we choose $y_0^\prime, y_1^\prime,\ldots,y_M^\prime$ to be a finite sequence in $\{y_j:\, j\in \Sigma_{M,0}\}$ such that
\begin{equation*}
  |y_{j}^\prime-y_{j-1}^\prime| > y_j-y_{j-1}, \quad \forall \, j\in \Sigma_M.
\end{equation*}
For each $j\in \Sigma_M$, we denote by $J'_j$ the closed interval with endpoints $y'_{j-1}$ and $y'_j$, and let $v_{j}:\, J_j^\prime\to J_j$ be a contractive homeomorphism with
\begin{equation}\label{eq:vj-cond}
 v_j(y_{j-1}^\prime)=y_{j-1}, \quad v_j(y_j^\prime)=y_j.
\end{equation}
Clearly, $v_j(y_j^\prime)=v_{j+1}(y_j^\prime)=y_j$ for all $j\in \Sigma_{M-1}$.


For all $i\in \Sigma_{N,0}$ and $j\in \Sigma_{M,0}$, we denote by $z'_{ij}$ to be the $z_{pq}$ satisfying $x_p=x'_i$ and $y_q=y'_j$.

For $i\in\Sigma_{N}$ and $j\in\Sigma_{M}$, we denote $D'_{ij}=I^\prime_{i}\times J^\prime_j$, and
define a continuous function $F_{ij}:\, D'_{ij}\times \mathbb{R}\rightarrow\mathbb{R}$ satisfying
\begin{equation}\label{eq: Fij-cond-1}
 F_{ij}(x'_{k},y'_{\ell},z'_{k\ell})=z_{k\ell}
\end{equation}
for all $(k,\ell)\in \{i-1,i\} \times \{j-1,j\}$, and
\begin{equation}\label{eq: Fij-cond-2}
 |F_{ij}(x,y,z^\prime)-F_{ij}(x,y,z^{\prime\prime})|\leq \alpha_{ij}|z'-z''|
\end{equation}
for all $(x,y)\in D^\prime_{ij}$ and all $z',z''\in\mathbb{R}$, where $\alpha_{ij}$ is a given constant with $0<\alpha_{ij}<1$.

Now we define a map $W_{ij}:D'_{ij}\times \mathbb{R}\rightarrow D_{ij}\times\mathbb{R}$ by
  \begin{equation}\label{wij}
    W_{ij}(x,y,z)=(u_{i}(x),v_{j}(y),F_{ij}(x,y,z)).
  \end{equation}
Then
\[
  W_{ij}(x'_k,y'_\ell, z'_{k\ell}) = (x_k,y_\ell, z_{k\ell}), \quad (k,\ell)\in \{i-1,i\} \times \{j-1,j\},
\]
and $\{D'_{ij}\times \mathbb{R}, W_{ij}; (i,j)\in\Sigma_{N}\times\Sigma_{M}\}$ is a \emph{recurrent iterated function system} (recurrent IFS).

For $i\in \Sigma_N$ and $j\in \Sigma_M$, we denote by $\mathcal{H}(D_{ij}\times \mathbb{R})$ the family of all non-empty compact subsets of $D_{ij}\times \mathbb{R}$. Let $\widetilde{\SH}$ be the product of all $\SH(D_{ij}\times \mathbb{R})$, i.e.,
\begin{equation*}
  \widetilde{\SH}=\{(A_{ij})_{1\leq i\leq N,1\leq j\leq N}:\, A_{ij}\in \SH(D_{ij}\times \mathbb{R}) \mbox{ for all } i,j\}.
\end{equation*}



For any $A=(A_{ij})\in \widetilde{\SH}$, we define $W(A)\in \widetilde{\SH}$ by
\begin{equation*}
  \big( W(A)\big)_{ij} = \cup\{ W_{ij}(A_{k\ell}):\, (k,\ell)\in \Sigma_N\times\Sigma_M \textrm{ and } D_{k\ell}\subset D_{ij}^\prime \}
\end{equation*}
for all $1\leq i\leq N$ and $1\leq j\leq M$.

For any continuous function $f$ on $I\times J$ and $(i,j)\in\Sigma_N\times\Sigma_M$,  we denote
  $$\Gamma f|_{D_{ij}}=\{(x,y,f(x,y)): (x,y)\in D_{ij}\},$$
which is the graph of $f$ restricted on $D_{ij}$. Similarly as in \cite{Bar86,BEH89,RuanXu15}, we have the following theorem.

\begin{theo}\label{theo: general-RFIF}
 Let $\{D'_{ij}\times \mathbb{R}, W_{ij}; (i,j)\in\Sigma_{N}\times\Sigma_{M}\}$ be the recurrent IFS defined by~\eqref{wij}. Assume that $\{F_{ij}: \, (i,j)\in\Sigma_N\times\Sigma_M\}$ satisfies the following matchable conditions:
\begin{itemize}
  \item[(1).] for all $i\in \Sigma_{N-1}, j\in\Sigma_{M}$, and $x^{*}=u_{i}^{-1}(x_{i})=u_{i+1}^{-1}(x_{i})$,
    \begin{equation}\label{eq: match-cond-1}
      F_{ij}(x^{*},y,z)=F_{i+1,j}(x^{*},y,z), \quad  \forall y\in J'_{j},\, z\in \mathbb{R}, \quad \mbox{and}
    \end{equation}
  \item[(2).] for all $i\in\Sigma_{N}, j\in \Sigma_{M-1}$, and $y^{*}=v_{j}^{-1}(y_{j})=v_{j+1}^{-1}(y_{j})$,
    \begin{equation}\label{eq: match-cond-2}
      F_{ij}(x,y^{*},z)=F_{i,j+1}(x,y^{*},z),  \quad \forall x\in I'_{i},\, z\in \mathbb{R}.
    \end{equation}
\end{itemize}
Then there exists a unique continuous function $f$ on $I\times J$, such that $f(x_{i},y_{j})=z_{ij}$ for all $i\in\Sigma_{N,0}$ and $j\in\Sigma_{M,0}$, and $( \Gamma f|_{D_{ij}})_{1\leq i\leq N, 1\leq j\leq M}$ is the invariant set of $W$, i.e., $\Gamma f|_{D_{ij}}=W_{ij}(\Gamma f|_{D_{ij}^\prime})$ for all $(i,j)\in\Sigma_N\times\Sigma_M$.
\end{theo}
\begin{proof}
Let $C^{*}(I\times J)$ be the collection of all continuous functions $\varphi$ on $I\times J$ satisfying $\varphi(x_{i},y_{j})=z_{ij}$ for all $(i,j)\in\Sigma_{N,0}\times\Sigma_{M,0}$. Define $T:\, C^{*}(I\times J)\to C^{*}(I\times J)$ as follows: given $\varphi\in C^{*}(I\times J)$,
\begin{align}\label{eq:T-def}
  T\varphi(x,y)=F_{ij}\big(u^{-1}_{i}(x),v^{-1}_{j}(y),\varphi(u^{-1}_{i}(x),v^{-1}_{j}(y))\big),\qquad (x,y)\in D_{ij},
\end{align}
for all $(i,j)\in\Sigma_{N}\times\Sigma_{M}$.

Given  $\varphi\in C^{*}(I\times J)$,  it is clear that for all $(i,j)\in\Sigma_N\times\Sigma_M$, $T\varphi|_{int(D_{ij})}$ is continuous, where we use $int(E)$ to denote the interior of a subset $E$ of $\bR^2$. For all $(x_{p},y_{q})\in D_{ij}$,
\begin{align*}
         &F_{ij}(u^{-1}_{i}(x_{p}),v^{-1}_{j}(y_{q}),\varphi(u^{-1}_{i}(x_{p}),v^{-1}_{j}(y_{q})))\\
               =\, &F_{ij}(x^\prime_{p},y^\prime_{q},\varphi(x^\prime_{p},y^\prime_{q}))
               =F_{ij}(x^\prime_{p},y^\prime_{q},z'_{pq})=z_{pq}
\end{align*}
so that $T\varphi(x_p,y_q)=z_{pq}$.
Furthermore, from matchable conditions, we know that $T\varphi$ is well defined on the boundary of $I_{i}\times J_{j}$ for all $(i,j)\in\Sigma_{N}\times\Sigma_{M}$. Thus $T:\, C^{*}(I\times J)\to C^{*}(I\times J)$ is well defined.

For any $\varphi\in C^{*}(I\times J)$, we define $|\varphi|_{\infty}=\max\{\varphi(x,y):\,(x,y)\in I\times J\}$. From \eqref{eq: Fij-cond-2}, we can easily see that $T$ is contractive on the complete metric space $(C^{*}(I\times J),|\cdot|_{\infty})$. Thus there exists a unique function $f\in C^{*}(I\times J)$ such that $Tf=f$, that is,
\begin{equation}\label{eq; f-recu-prop-proof}
f(x,y)=F_{ij}\big(u_{i}^{-1}(x),v_{j}^{-1}(y),f(u_{i}^{-1}(x),v_{j}^{-1}(y))\big),\qquad (x,y)\in D_{ij},
\end{equation}
for all $(i,j)\in\Sigma_{N}\times\Sigma_{M}$. Combining this with \eqref{wij}, we know that for all $(i,j)\in\Sigma_{N}\times\Sigma_{M}$,
\begin{align*}
\Gamma f|_{D_{ij}}&=\{(x,y,f(x,y)):\,(x,y)\in D_{ij}\}\\
           &=\{(x,y,F_{ij}\big(u_{i}^{-1}(x),v_{j}^{-1}(y),f(u_{i}^{-1}(x),v_{j}^{-1}(y))\big)):\,(x,y)\in D_{ij}\}\\
           &=\{(u_{i}(x),v_{j}(y),F_{ij}(x,y,f(x,y))):\,(x,y)\in D^\prime_{ij}\}\\
           &=\{W_{ij}(x,y,f(x,y)):\,(x,y)\in D^\prime_{ij}\}.
\end{align*}
Thus $( \Gamma f|_{D_{ij}})_{1\leq i\leq N, 1\leq j\leq M}$ is the invariant set of $W$.

Assume that there exists $\widetilde{f}\in C^{*}(I\times J)$, such that $( \Gamma \widetilde{f}|_{D_{ij}})_{1\leq i\leq N, 1\leq j\leq M}$ is the invariant set of $W$. Then we must have
\begin{equation*}
  F_{ij}(x,y,\widetilde{f}(x,y)) = \widetilde{f} (u_i(x),v_j(y)), \qquad \forall (x,y)\in D_{ij}^\prime,
\end{equation*}
so that $T\widetilde{f}=\widetilde{f}$. Since $T$ is contractive on $(C^{*}(I\times J),|\cdot|_{\infty})$, we know that $\widetilde{f}=f$.
\end{proof}


We call $f$ the \emph{recurrent fractal interpolation function} (RFIF) defined by $\{D'_{ij}\times \mathbb{R}, W_{ij}; (i,j)\in\Sigma_{N}\times\Sigma_{M}\}$, and call the graph of $f$ a \emph{recurrent fractal interpolation surface} (RFIS). From \eqref{eq; f-recu-prop-proof}, we have the following useful property: for all $(i,j)\in \Sigma_N\times \Sigma_M$,
\begin{equation}\label{eq: f-recu-prop}
  f(u_i(x),v_j(y)) = F_{ij} (x,y,f(x,y)) \quad \textrm{for all $(x,y)\in D_{ij}'$}.
\end{equation}

Generally, the recurrent IFS $\{D'_{ij}\times \mathbb{R}, W_{ij};(i,j)\in\Sigma_{N}\times\Sigma_{M}\}$ is not hyperbolic. However, we can still show that $\Gamma f=\{(x,y,f(x,y)):(x,y)\in I\times J\}$ is the attractor of the recurrent IFS. The spirit of the proof follows from \cite{Bar86,BEH89,RuanXu15}. For any two nonempty compact subsets $A$ and $B$ of $\mathbb{R}^{3}$, we define their Hausdorff metric by
\begin{equation*}
d_{H}(A,B)=\max\{\max_{x\in A}\min_{y\in B} d(x,y),\max_{y\in B}\min_{x\in A} d(x,y)\},
\end{equation*}
where we use $d(\cdot ,\cdot)$ to denote the Euclidean metric. For any $A,B\in \widetilde{\SH}$ with $A=(A_{ij})_{1\leq i\leq N, 1\leq j\leq M}$ and $B=(B_{ij})_{1\leq i\leq N, 1\leq j\leq M}$, we define
\begin{equation*}
  \widetilde{d}_{H}(A,B)=\max\{d_{H}(A_{ij},B_{ij}):\, (i,j)\in\Sigma_{N}\times\Sigma_{M}\}.
\end{equation*}
It is well known that $(\widetilde{\SH},\widetilde{d}_{H})$ is a complete metric space. For details about the Hausdorff metric and the metric $\widetilde{d}_{H}$, please see \cite{Bar93,BEH89,Fal90}.
\begin{theo}
Let $f$ be the RFIF defined by $\{D'_{ij}\times \mathbb{R}, W_{ij}; (i,j)\in\Sigma_{N}\times\Sigma_{M}\}$. Then for any $A=(A_{ij})_{1\leq i\leq N, 1\leq j\leq M}\in \widetilde{\SH}$ with $A_{ij}\neq\emptyset$ for all $(i,j)\in \Sigma_N\times \Sigma_M$,
\begin{equation}\label{eq:Haud-Dist-conv}
  \lim_{n\to\infty}\widetilde{d}_{H}(W^{n}(A),( \Gamma f|_{D_{ij}})_{1\leq i\leq N, 1\leq j\leq M})=0,
\end{equation}
where $W^{0}(A)=A$ and $W^{n+1}(A)=W(W^{n}(A))$ for any $n\geq0$.
\end{theo}
\begin{proof}
For any $A=(A_{ij})_{1\leq i\leq N, 1\leq j\leq M}\in \widetilde{\SH}$, we define $A_{XY}$ to be the projection of $A$ onto $I\times J$, i.e.,
\begin{align*}
A_{XY}=\{(x,y):\,\mbox{there exists } (i,j)\in \Sigma_N\times \Sigma_M \ \mbox{and}\ z\in\mathbb{R}  \mbox{ with}\ (x,y,z)\in A_{ij}\}.
\end{align*}
Let $\SH(I\times J)$ be the family of all nonempty compact subsets of $I\times J$. It is clear that $A_{XY}\in \SH(I\times J)$.

Define $L:\,\SH(I\times J)\to \SH(I\times J)$ by
\begin{align*}
L(U)=\bigcup_{\substack{1\leq i\leq N\\ 1\leq j\leq M}}\{(u_{i}(x),v_{j}(y)):\,(x,y)\in D_{ij}^{\prime}\cap U\},\quad U\in\SH(I\times J).
\end{align*}
It is clear that $L$ is contractive on $\SH(I\times J)$. Since $A_{ij}\neq\emptyset$ for all $(i,j)\in \Sigma_N\times \Sigma_M$, we have
\begin{align*}
\lim_{n\to\infty}d_{H}(L^{n}(A_{XY}),I\times J)=0.
\end{align*}
Let $\Omega_{ij}^{n}=\{(x,y,f(x,y)):\,(x,y)\in L^{n}(A_{XY})\cap D_{ij}\}$. Noticing that $f$ is uniformly continuous on $I\times J$, we have
\begin{align}\label{eq:Omega}
\lim_{n\to\infty}d_{H}(\Omega_{ij}^{n}, \Gamma f|_{D_{ij}})=0
\end{align}
for all $(i,j)\in \Sigma_N\times \Sigma_M$.

For $n\geq 1$, we define
\begin{align*}
\Delta_{n}=\sup\{|f(x,y)-z|:\,(x,y,z)\in(W^{n}(A))_{ij} \mbox{ for some } (i,j)\in \Sigma_N\times \Sigma_M\}.
\end{align*}
By the definition of Hausdorff metric, we have $d_{H}(\Omega^{n}_{ij},(W^{n}(A))_{ij})\leq \Delta_{n}$ for all $(i,j)\in \Sigma_N\times \Sigma_M$ and $n\geq 1$.

Given $(i,j)\in \Sigma_N\times \Sigma_M$ and $n\geq 1$, for any $(x,y,z)\in (W^{n}(A))_{ij}$, there exist $(k,\ell)\in \Sigma_N\times \Sigma_M$ and $(x^*,y^*,z^*)\in (W^{n-1}(A))_{k\ell}$ such that
\begin{align*}
(x,y,z)=(u_{i}(x^*),v_{j}(y^*),F_{ij}(x^*,y^*,z^*)).
\end{align*}
Let $\alpha=\max\{\alpha_{ij}:\,(i,j)\in \Sigma_N\times \Sigma_M\}$. By \eqref{eq: Fij-cond-2} and \eqref{eq: f-recu-prop},
 \begin{align*}
|f(x,y)-z|&=|F_{ij}(x^*,y^*,f(x^*,y^*))-F_{ij}(x^*,y^*,z^*))|\\
          &\leq \alpha|f(x^*,y^*)-z^*|\leq \alpha \Delta_{n-1}
\end{align*}
so that $\Delta_{n}\leq\alpha\Delta_{n-1}$. Since $0<\alpha<1$, we have $\lim\limits_{n\to\infty}\Delta_{n}=0$. Thus
\begin{align*}
\lim_{n\to\infty}d_{H}(\Omega^{n}_{ij},(W^{n}(A))_{ij})=0
\end{align*}
for all $(i,j)\in \Sigma_N\times \Sigma_M$. Combining this with \eqref{eq:Omega}, we have
\begin{align*}
\lim_{n\to\infty}d_{H}((W^{n}(A))_{ij}, \Gamma f|_{D_{ij}})=0
\end{align*}
for all $(i,j)\in \Sigma_N\times \Sigma_M$. This completes the proof of the theorem.
\end{proof}

\begin{rem}
  The assumption ``$A_{ij}\neq\emptyset$ for all $(i,j)\in \Sigma_N\times \Sigma_M$" in the above theorem can be weaken.
  In fact, from the proof, we can see that the theorem still holds if we only require that  $\lim_{n\to\infty}d_{H}(L^{n}(A_{XY}),I\times J)=0$.
\end{rem}

\section{Construction of Bilinear RFISs}
\label{sec:bilinear}

Let $h(x,y), S(x,y)$ be two continuous functions on $I\times J$ satisfying
\begin{align}
  h(x_{i},y_{j})=z_{ij},  \qquad (i,j)\in\Sigma_{N,0}\times\Sigma_{M,0}, \quad \textrm{and}\label{eq:h-cond}
\end{align}
$$\max\{|S(x,y)|:(x,y)\in I\times J\}<1.$$

Let $g_{ij}(x,y), (i,j)\in \Sigma_N\times \Sigma_M$ be continuous functions on $D'_{ij}$ satisfying
\begin{equation}\label{eq:gij-cond}
  g_{ij}(x'_{k},y'_{\ell})=z'_{k\ell}, \qquad (k,\ell)\in\{i-1,i\}\times\{j-1,j\}.
\end{equation}

For $(i,j)\in\Sigma_{N}\times\Sigma_{M}$, we define functions $F_{ij}(x,y,z):\, D'_{ij}\times \mathbb{R}\to \mathbb{R}$ by
\begin{equation}\label{eq: Fij-def}
  F_{ij}(x,y,z)=S(u_{i}(x),v_{j}(y))(z-g_{ij}(x,y))+h(u_{i}(x),v_{j}(y)),
\end{equation}
where $u_i\in C(I_i^\prime)$ and $v_j\in C(J_j^\prime)$ are functions defined in Section~2. Then for all $(i,j)\in \Sigma_N\times \Sigma_M$ and $(k,\ell)\in\{i-1,i\}\times\{j-1,j\}$,
\begin{equation*}
  F_{ij}(x'_k,y'_\ell,z'_{k\ell})=h(u_i(x'_k),v_j(y'_\ell))=h(x_k,y_\ell)=z_{k\ell}
\end{equation*}
so that \eqref{eq: Fij-cond-1} holds. Furthermore, it follows from $\max\{|S(x,y)|:(x,y)\in I\times J\}<1$ that \eqref{eq: Fij-cond-2} also holds.

Given $i\in \Sigma_{N-1}$, let $x^{*}=u_{i}^{-1}(x_{i})=u_{i+1}^{-1}(x_{i})$. Given $j\in\Sigma_{M}$, for all $y\in J'_{j}$ and $z\in \mathbb{R}$, we have
\begin{align*}
  & F_{ij}(x^{*},y,z)=S(x_i,v_j(y))(z-g_{ij}(x^*,y))+h(x_i,v_j(y)), \\
  & F_{i+1,j}(x^{*},y,z)=S(x_i,v_j(y))(z-g_{i+1,j}(x^*,y))+h(x_i,v_j(y)).
\end{align*}
Thus, the matchable condition \eqref{eq: match-cond-1} holds if $g_{ij}(x^*,y)=g_{i+1,j}(x^*,y)$ for all $y\in J_j^\prime$. Similarly, the matchable condition \eqref{eq: match-cond-2} holds if $g_{ij}(x,y^*)=g_{i,j+1}(x,y^*)$ for all $x\in I_i^\prime$. Thus, from Theorem~\ref{theo: general-RFIF}, we have the following result. We remark that Metzler and Yun~\cite{MeYu10}, and Yun, Choi and O \cite{YCO15} obtained similar results.



\begin{theo}\label{theo:gij}
Let $F_{ij}, (i,j)\in\Sigma_{N}\times\Sigma_{M}$ be defined as in \eqref{eq: Fij-def}.  Assume that the following matchable conditions hold:
\begin{itemize}
  \item[(1).] for all $i\in \Sigma_{N-1}, j\in\Sigma_{M}$, and $x^{*}=u_{i}^{-1}(x_{i})=u_{i+1}^{-1}(x_{i})$,
    \begin{equation}\label{eq: gij-match-cond-1}
      g_{ij}(x^{*},y)=g_{i+1,j}(x^{*},y), \quad  \forall y\in J'_{j}, \quad \mbox{and}
    \end{equation}
  \item[(2).] for all $i\in\Sigma_{N}, j\in \Sigma_{M-1}$, and $y^{*}=v_{j}^{-1}(y_{j})=v_{j+1}^{-1}(y_{j})$,
    \begin{equation}\label{eq: gij-match-cond-2}
      g_{ij}(x,y^{*})=g_{i,j+1}(x,y^{*}),  \quad \forall x\in I'_{i}.
    \end{equation}
\end{itemize}
Then there exists a unique continuous function $f$ on $I\times J$, such that $( \Gamma f|_{D_{ij}})_{1\leq i\leq N, 1\leq j\leq M}$ is the invariant set of $W$ with $f(x_{i},y_{j})=z_{ij}$ for all $i\in\Sigma_{N,0}$ and $j\in\Sigma_{M,0}$.
\end{theo}

We remark that from \eqref{eq:gij-cond}, it is easy to check that for all $i\in \Sigma_{N-1}$, $j\in \Sigma_M$, and $x^{*}=u_{i}^{-1}(x_{i})=u_{i+1}^{-1}(x_{i})$,
\begin{equation}\label{eq:gij-prop-1}
  g_{ij}(x^*,y'_{j-1})=g_{i+1,j}(x^*,y'_{j-1}), \quad g_{ij}(x^*,y'_j)=g_{i+1,j}(x^*,y'_j).
\end{equation}
Similarly, for all $i\in\Sigma_{N}, j\in \Sigma_{M-1}$, and $y^{*}=v_{j}^{-1}(y_{j})=v_{j+1}^{-1}(y_{j})$,
\begin{equation}\label{eq:gij-prop-2}
  g_{ij}(x'_{i-1},y^*)=g_{i,j+1}(x'_{i-1},y^*), \quad g_{ij}(x'_{i},y^*)=g_{i,j+1}(x'_{i},y^*).
\end{equation}
However, the match conditions in Theorem~\ref{theo:gij} may not be satisfied if we only require that \eqref{eq:gij-prop-1} and \eqref{eq:gij-prop-2} hold.

Now we want to construct bilinear RFISs which are special RFISs in the above theorem. The basic idea is similar to that in \cite{BaMa15,RuanXu15}.  For all $i\in \Sigma_N$ and $j\in\Sigma_M$, we define $u_i$ and $v_j$ to be linear functions satisfying  \eqref{eq:ui-cond} and \eqref{eq:vj-cond}. We also define $g_{ij}$ to be the bilinear function satisfying \eqref{eq:gij-cond}. That is, if we denote
\begin{equation*}
  \lambda_i(x)=\frac{x'_{i}-x}{x'_{i}-x'_{i-1}}, \qquad \mu_j(y)=\frac{y'_{j}-y}{y'_{j}-y'_{j-1}},
\end{equation*}
then for all $(x,y)\in D_{ij}^\prime$,
\begin{align*}
  g_{ij}(x,y)=&\lambda_i(x) \mu_j(y) z'_{i-1,j-1} + (1-\lambda_i(x)) \mu_j(y) z'_{i,j-1} \\
  &+\lambda_i(x)(1- \mu_j(y)) z'_{i-1,j} + (1-\lambda_i(x))(1- \mu_j(y)) z'_{i,j}.
\end{align*}
Notice that once $x$ is fixed, $g_{ij}(x,y)$ is a linear function of $y$. Combining this with \eqref{eq:gij-prop-1}, the matchable condition \eqref{eq: gij-match-cond-1} in Theorem~\ref{theo:gij} is satisfied. Similarly, \eqref{eq: gij-match-cond-2} is also satisfied. Thus, in the case that $g_{ij}$ are bilinear for all $i,j$, we do generate RFIF and RFIS.

In order to construct bilinear RFISs, we define $h:\, I\times J\to \mathbb{R}$ to be the function satisfying \eqref{eq:h-cond} and $h|_{I_i\times J_j}$ is bilinear for all $(i,j)\in\Sigma_N\times \Sigma_M$.

Let $\{s_{ij}:\, (i,j)\in \Sigma_{N,0}\times \Sigma_{M,0}\}$ be a given subset of $\mathbb{R}$ with $|s_{ij}|<1$ for all $i,j$. We define $S:\, I\times J\to \mathbb{R}$ to be the function such that $S|_{I_i\times J_i}$ is bilinear for all $(i,j)\in \Sigma_N\times \Sigma_M$ and
\begin{equation*}
  S(x_i,y_j) = s_{ij}, \qquad \forall (i,j)\in \Sigma_{N,0}\times \Sigma_{M,0}.
\end{equation*}

For all $(i,j)\in\Sigma_N\times \Sigma_M$, we define $F_{ij}:\, D_{ij}^\prime\times \mathbb{R}\to \mathbb{R}$ by \eqref{eq: Fij-def}. Then the RFIF $f$ is called a \emph{bilinear RFIF}. We call $\Gamma f$ a \emph{bilinear RFIS}. The function $S(x,y)$ is called the \emph{vertical scale factor function} of $f$. We also call $s_{ij}, (i,j)\in \Sigma_{N,0}\times \Sigma_{M,0}$ \emph{vertical scaling factors} of $f$ if there is no confusion.

From the construction, we know that a bilinear RFIS is determined by interpolation points $\{(x_i,y_j,z_{i,j})\}_{(i,j)\in \Sigma_{N,0}\times \Sigma_{M,0}}$, vertical scaling factors $\{s_{ij}:\, (i,j)\in \Sigma_{N,0}\times \Sigma_{M,0}\}$, and domain points $\{(x_i^\prime,y_j^\prime)\}_{(i,j)\in \Sigma_{N,0}\times \Sigma_{M,0}}$. This property is similar to the linear recurrent FIF in the one-dimensional case \cite{BEH89}. In particular, a bilinear RFIS is easy to be generated, while there are no restrictions on interpolation points and vertical scaling factors.

\section{Box Dimension of Bilinear RFISs}
For any $k_1, k_2, k_3\in\mathbb{Z}$ and $\varepsilon>0$, we call $\Pi_{i=1}^{3}[k_i\varepsilon,(k_i+1)\varepsilon]$ an
$\varepsilon$-coordinate cube in $\mathbb{R}^3$. Let $E$ be a bounded set in $\mathbb{R}^3$ and $\mathcal{N}_E(\varepsilon)$ the number of
$\varepsilon$-coordinate cubes intersecting $E$. We define
\begin{equation}\label{eq:box-dim-def}
  \overline{\dim}_{B} E=\varlimsup_{\gep\to 0+}\frac{\log \mathcal{N}_{E}(\gep)}{\log1/\gep}\quad\text{and}\quad  \underline{\dim}_{B} E=\varliminf_{\gep\to 0+}\frac{\log \mathcal{N}_{E}(\gep)}{\log1/\gep},
\end{equation}
and call them the \emph{upper box dimension} and the \emph{lower box dimension} of $E$, respectively.
If $\overline{\dim}_{B} E=\underline{\dim}_{B} E$, then we use  $\dim_B E$ to denote the common value and  call it the \emph{box dimension} of $E$.
It is easy to see  that in the definition of the upper and lower box dimensions, we can only consider $\gep_n=\frac{1}{K^nN}$, where $n\in \mathbb{Z}^+$. That is,
\begin{equation}\label{eq:box-dim-def2}
  \overline{\dim}_{B} E=\varlimsup_{n \to \infty}\frac{\log \mathcal{N}_{E}(\gep_{n})}{n\log K}\quad\text{and}\quad  \underline{\dim}_{B} E=\varliminf_{n\to \infty}\frac{\log \mathcal{N}_{E}(\gep_{n})}{n\log K}.
\end{equation}
It is also well known that $\underline{\dim}_{B}E\geq 2$ when $E$ is the graph of a continuous function on a domain of $\mathbb{R}^{2}$. Please see ~\cite{Fal90} for details.

In this section, we will estimate the box dimension of $\Gamma f$, where $f$ is the bilinear RFIF defined in the section~\ref{sec:bilinear}.
Without loss of generality, we can assume that $I=J=[0,1]$.

\subsection{A method to calculate the box dimension}
It is difficult to obtain the box dimension of general bilinear RFIS.
In this paper, we assume that
\begin{equation}\label{eq:hom-cond-1}
  M=N,  \quad\mbox{and } x_{i}=\frac{i}{N},  y_j=\frac{j}{N}, \quad \forall i,j\in\Sigma_{N,0}.
\end{equation}
Furthermore, we assume that there exists a positive integer $K\geq 2$ such that
 \begin{align}\label{eq:hom-cond-2}
   \frac{|I_i^\prime|}{|I_i|} =\frac{|J_j^\prime|}{|J_j|}=K, \qquad \forall i,j\in \Sigma_N,
 \end{align}
and for all $(i_1,j_1), (i_2, j_2) \in \Sigma_N\times \Sigma_N$, we have
  \begin{align}\label{eq:hom-cond-3}
   D_{i_1 j_1}^\prime = D_{i_2 j_2}^\prime \quad \mbox{or} \quad  int(D_{i_1 j_1}^\prime) \cap int(D_{i_2 j_2}^\prime) = \emptyset.
 \end{align}

For each $n\in \mathbb{Z}^+$ and $1\leq k,\ell\leq K^nN$, we denote
  $$D_{k\ell}^n = \left[\frac{k-1}{K^n N},\frac{k}{K^n N}\right] \times  \left[\frac{\ell-1}{K^n N},\frac{\ell}{K^n N}\right].$$
Given $n\in \bZp$ and $U\subset [0,1]^2$, we define
\begin{equation*}
  O(f,n,U) = \sum_{\substack{1\leq k,\ell\leq K^nN \\ D_{k\ell}^n\subset U}} O(f,D_{k\ell}^n),
\end{equation*}
where we use $O(f,E)$ to denote the oscillation of $f$ on $E\subset [0,1]^2$, that is,
\begin{equation*}
 O(f,E) = \sup\{f(\bfx^\prime)-f(\bfx^{\prime\prime}):\; \bfx^\prime,\bfx^{\prime\prime}\in E\}.
\end{equation*}
We also denote $O(f,n)=O(f,n,[0,1]^2)$ for simplicity.

We will use the following simple lemma, which presents a method to estimate the upper and lower box dimensions of the graph of a function from its oscillation. Similar results can be found in \cite{Fal90,KRZ18,RSY09}.

\begin{lem}\label{lem:box-dim}
Let $f$ be the bilinear RFIF defined in Section~3. Then
\begin{equation}\label{eq:lower-boxdim-est}
  \underline{\dim}_B(\Gamma f) \geq \max\Big\{2,1+\varliminf_{n\to\infty} \frac{\log O(f,n)}{n\log K} \Big\}, \quad \textrm{and}
\end{equation}
\begin{equation}\label{eq:upper-boxdim-est}
  \overline{\dim}_B(\Gamma f) \leq 1+\varlimsup_{n\to\infty} \frac{\log (O(f,n)+2K^nN)}{n\log K},
\end{equation}
where we define $\log0=-\infty$ according to the usual convention.
\end{lem}

\begin{proof}
  It is clear that
  \begin{align}
    &\N_{\Gamma f}(\gep_n) \geq \gep_n^{-1}  \sum_{1\leq i,j\leq K^nN} O(f,D_{ij}^n)=\gep_n^{-1}O(f,n)
  \end{align}
  so that \eqref{eq:lower-boxdim-est} holds. On the other hand, we note that $\N_E(\gep)$ and $\N_E(\gep_n)$ can be replaced by $\wdt{\N}_E(\gep)$ and $\wdt{\N}_E(\gep_n)$ in \eqref{eq:box-dim-def} and \eqref{eq:box-dim-def2} respectively, where $\wdt{N}_E(\gep)$ is the smallest number of cubes of side $\gep$ that cover $E$ (see \cite{Fal90} for details). In our case,
  \begin{equation*}
    \wdt{N}_{\Gamma f} (\gep_n) \leq \sum_{1\leq i,j\leq K^nN} (\gep_n^{-1}O(f,D_{ij}^n)+2) = K^n N\big( O(f,n)+2K^nN\big)
  \end{equation*}
  so that \eqref{eq:upper-boxdim-est} holds.
\end{proof}



\begin{rem}
  From Lemma~\ref{lem:box-dim}, $\dim_B(\Gamma f) =2$ if $\varlimsup\limits_{n\to\infty} \frac{\log O(f,n)}{n\log K}\leq 1$, and $\dim_B(\Gamma f) =1+\lim\limits_{n\to\infty}\frac{\log O(f,n)}{n\log K}$ if the limit exists and is larger than $1$.
\end{rem}

\subsection{Compatible partitions and uniform sums}
The vertical scaling factors $\{s_{ij}:\, i,j\in \Sigma_{N,0}\}$ are called \emph{steady} if for each $(i,j)\in \Sigma_N\times \Sigma_N$, either all of $s_{i-1,j-1},s_{i-1,j},s_{i,j-1}$ and $s_{ij}$ are nonnegative or all of them are nonpositive.

Given $\SB=\{B_r\}_{r=1}^m$, where $B_1,\ldots,B_m$ are nonempty subsets of $[0,1]^2$, $\SB$ is called a \emph{partition} of $[0,1]^2$ if $\bigcup_{r=1}^m B_r=[0,1]^2$ and $int(B_r)\cap int(B_t)=\emptyset$ for all $r\not=t$. A partition $\SB=\{B_r\}_{r=1}^m$ of $[0,1]$ is called \emph{compatible} with respect to $\{D_{ij},D_{ij}^\prime;(i,j)\in \Sigma_N\times \Sigma_M\}$ if the following two conditions hold:
\begin{enumerate}
  \item for each $(i,j)\in \Sigma_N\times \Sigma_N$, there exist $r,t\in\{1,2,\ldots,m\}$, such that $D_{ij}\subset B_r$ and $D_{ij}^\prime\subset B_t$, and
 \item  Assume that $r,t\in \{1,2,\ldots,m\}$. If there exists $(i,j)\in \Sigma_N\times \Sigma_N$, such that $D_{ij}\subset B_r$ and $D'_{ij}\subset B_t$, then
 \[
   B_t = \bigcup \{D'_{k\ell}:\, D_{k\ell}\subset B_r, D'_{k\ell}\subset B_t\}.
 \]
\end{enumerate}

  For $1\leq r,t\leq m$, we denote
\begin{align*}
  & \Lambda_r=\{(i,j)\in \Sigma_N\times \Sigma_N:\, D_{ij}\subset B_r\},\\
  & \Lambda_t^\prime=\{(i,j)\in \Sigma_N\times \Sigma_N:\, D_{ij}^\prime\subset B_t\},
\end{align*}
We remark that the second condition is equivalent to
 \[
   B_t = \bigcup_{(i,j)\in \Lambda_r\cap \Lambda_t^\prime} D_{ij}^\prime, \quad \mbox{ if } \Lambda_{r}\cap \Lambda_t^\prime\not=\emptyset.
 \]








Given $1\leq r\leq m$ and $\alpha,\beta\in \{0,1,\ldots,N-K\}$, we denote
 $$\Lambda_r(\alpha,\beta)= \{(i,j)\in \gL_r:\, D_{ij}^\prime=[x_\alpha,x_{\alpha+K}] \times [y_\beta,y_{\beta+K}]\}.$$
We say that the vertical scaling factors $\{s_{ij}:\, i,j\in \Sigma_{N,0}\}$ \emph{have uniform sums} under a compatible partition $\SB=\{B_r\}_{r=1}^m$ if for all $r,t\in \{1,2,\ldots,m\}$ with $\Lambda_r \cap \Lambda_t^\prime\not=\emptyset$, there exists a constant $\gamma_{rt}$, such that
\begin{align*}
   \gamma_{rt} & =\sum_{(i,j)\in \Lambda_r(\alpha,\beta)} |S(u_i(x_{\alpha}),v_{j}(y_{\beta}))|
                 =\sum_{(i,j)\in \Lambda_r(\alpha,\beta)} |S(u_i(x_\alpha),v_{j}(y_{\beta+K}))|   \\
                & =\sum_{(i,j)\in \Lambda_r(\alpha,\beta)} |S(u_i(x_{\alpha+K}),v_{j}(y_{\beta}))|
                 =\sum_{(i,j)\in \Lambda_r(\alpha,\beta)} |S(u_i(x_{\alpha+K}),v_{j}(y_{\beta+K}))|
\end{align*}
for all $\alpha,\beta\in \{0,1,\ldots,N-K\}$ with $[x_\alpha,x_{\alpha+K}] \times [y_\beta,y_{\beta+K}]\in \{D_{ij}^\prime:\, (i,j)\in \gL_r\cap \gL'_t\}$.
In this case, we also call $\{\gamma_{rt}:\, \Lambda_r\cap \Lambda_t^\prime\not=\emptyset\}$ the \emph{uniform sums} of vertical scaling factors.

\begin{lem}\label{lem:vertScal}
  Assume that the vertical scaling factors $\{s_{ij}:\, i,j\in \Sigma_{N,0}\}$ are steady and have uniform sums $\{\gamma_{rt}:\, \Lambda_r\cap \Lambda_t^\prime\not=\emptyset\}$ under a compatible partition $\SB=\{B_r\}_{r=1}^m$.  Then for all $r,t\in \{1,2,\ldots,m\}$ with $\Lambda_r \cap \Lambda_t^\prime\not=\emptyset$, and all $\alpha,\beta\in \{0,1,\ldots,N-K\}$ with $[x_\alpha,x_{\alpha+K}] \times [y_\beta,y_{\beta+K}]\in \{D_{ij}^\prime:\, (i,j)\in \gL_r\cap \gL'_t\}$, we have
  \begin{equation*}
    \sum_{(i,j)\in \Lambda_r(\alpha,\beta)} |S(u_i(x),v_{j}(y))|=\gamma_{rt}, \quad \forall (x,y)\in [x_\alpha,x_{\alpha+K}] \times [y_\beta,y_{\beta+K}].
  \end{equation*}
\end{lem}
\begin{proof}
   Given $r,t\in \{1,2,\ldots,m\}$ with $\Lambda_r \cap \Lambda_t^\prime\not=\emptyset$, and given $\alpha,\beta\in \{0,1,\ldots,N-K\}$ with $[x_\alpha,x_{\alpha+K}] \times [y_\beta,y_{\beta+K}]\in \{D_{ij}^\prime:\, (i,j)\in \gL'_t\}$, we define
  \begin{equation*}
    S^*(x,y)=\sum_{(i,j)\in \Lambda_r(\alpha,\beta)}  |S(u_i(x),v_j(y))| - \gamma_{rt},
  \end{equation*}
  where $(x,y)\in [x_\alpha,x_{\alpha+K}] \times [y_\beta,y_{\beta+K}]$.
  Since the vertical scaling factors are steady, we know that for each $(i,j)\in \Lambda_r(\alpha,\beta)$, the function $S$ is nonnegative or nonpositive on $D_{ij}$.
  It follows that $S(u_i(x),v_j(y))$ is nonnegative or nonpositive on $[x_\alpha,x_{\alpha+K}] \times [y_\beta,y_{\beta+K}]$. As a result,  $S^*$ is a bilinear function on $[x_\alpha,x_{\alpha+K}] \times [y_\beta,y_{\beta+K}]$. Thus, from $S^*=0$ on $(x_\alpha,y_\beta),(x_{\alpha},y_{\beta+K}), (x_{\alpha+K},y_{\beta})$ and $(x_{\alpha+K},y_{\beta+K})$, we know that $S^*=0$ on $[x_\alpha,x_{\alpha+K}] \times [y_\beta,y_{\beta+K}]$ so that the lemma holds.
\end{proof}





\subsection{Calculation of box dimension}
In this subsection, we always assume that $f$ is the bilinear RFIF determined in the section~\ref{sec:bilinear} with conditions \eqref{eq:hom-cond-1}, \eqref{eq:hom-cond-2} and \eqref{eq:hom-cond-3}. Furthermore, we assume that the vertical scaling factors $\{s_{ij}:\, i,j\in \Sigma_{N,0}\}$ are steady and have uniform sums $\{\gamma_{rt}:\, \Lambda_r\cap \Lambda_t^\prime\not=\emptyset\}$ under a compatible partition $\SB=\{B_r\}_{r=1}^m$.

Using our assumptions, we can obtain the following basic result.
\begin{lem}\label{lem:Osi-local-ineq-1}
  There exists a positive constant $C>0$ such that
  \begin{equation}\label{eq:N-rec-rel}
      \Big|O(f,n+1,B_{r})-\sum_{t=1}^m \gamma_{rt} O(f,n,B_{t})\Big|\leq CK^{n}
  \end{equation}
  for all $1\leq r\leq m$ and $n\in \mathbb{Z}^+$, where we define $\gamma_{rt}=0$ if $\Lambda_r\cap \Lambda_t^\prime=\emptyset$.
\end{lem}
\begin{proof}
%
%

From the first condition of compatible partition,
\begin{equation*}
  B_r=\bigcup_{t:\, \Lambda_r\cap \Lambda_t'\not=\emptyset} \cup\{D_{ij}:\, (i,j)\in \Lambda_r\cap \Lambda_t'\}
\end{equation*}
for all $1\leq r\leq m$. Thus, in order to prove the lemma, it suffices to show that for all $1\leq r,t\leq m$ with $\Lambda_r\cap \Lambda_t^\prime\not=\emptyset$, there exists a constant $C_{rt}>0$, such that
  \begin{equation*}
      \Big| \sum_{(i,j)\in \Lambda_r\cap \Lambda_t^\prime} O(f,n+1,D_{ij})-  \gamma_{rt} O(f,n,B_{t})\Big|\leq C_{rt}K^{n}.
  \end{equation*}


For all $\alpha,\beta\in\{0,1,\ldots,N-K\}$, we denote $\wdt{D}_{\ga\gb}=[x_\ga,x_{\ga+K}]\times [y_\gb,y_{\gb+K}]$.
From the second condition of compatible partition, for all $r,t=1,2,\ldots,m$ with $\Lambda_r\cap \Lambda_t^\prime\not=\emptyset$,
\begin{equation*}
  B_t=\cup\{ \wdt{D}_{\alpha\beta}:\, \exists (i,j)\in \Lambda_r\cap \Lambda_t^\prime, \textrm{ such that } \wdt{D}_{\alpha\beta}=D_{ij}^\prime\}.
\end{equation*}
On the other hand, it is clear that for each $(i,j)\in \Lambda_r\cap \Lambda_t^\prime$, there exists a unique $(\alpha,\beta)\in \{0,1,\ldots,N-K\}^2$, such that
$D_{ij}^\prime=\wdt{D}_{\alpha\beta}\subset B_t$.
Thus, in order to prove the lemma, it suffices to show that for all $r,t=1,2,\ldots,m$ with $\Lambda_r\cap \Lambda_t^\prime\not=\emptyset$, and all $\alpha,\beta=0,1,\ldots,N-K$ with $\wdt{D}_{\alpha\beta}=D_{ij}^\prime$ for some $(i,j)\in \Lambda_r\cap \Lambda_t^\prime$, there exists a constant $C_{r,t,\ga,\gb}>0$, such that
\begin{equation}\label{eq: osi-ineq-temp}
   \Big|\sum_{(i,j)\in \Lambda_r(\alpha,\beta)} O(f,n+1,D_{ij})- \gamma_{rt} O(f,n,\widetilde{D}_{\alpha,\beta})\Big|\leq C_{r,t,\ga,\gb} K^{n}
\end{equation}
for all positive integers $n$.

 Denote $M^*=\max\{|f(x,y)|:\,(x,y)\in [0,1]^2\}$. For $(i,j)\in \Sigma_N\times \Sigma_N$ and  $z^*\in[-M^*,M^*]$, we define
\begin{equation*}
  \hat{F}_{i,j,z^*}(x,y) = F_{ij}(x,y,z^*), \quad (x,y)\in D_{ij}^\prime.
\end{equation*}
Since $S(u_i(x),v_j(y)), g_{ij}(x,y), h(u_i(x),v_j(y))$ are all bilinear functions on $D_{ij}^\prime$, we know that
\begin{equation*}
  C_{ij} = \sup_{\substack{(x,y)\in int(D_{ij}^\prime) \\ z^*\in [-M^*,M^*]}}
  \|\nabla \hat{F}_{i,j,z^*}(x,y) \| <\infty,
\end{equation*}
where $\| \cdot \|$ is the standard Euclidean norm.

Given $1\leq k,\ell\leq K^{n}N$ with $D_{k\ell}^n \subset D_{ij}'$, we fix a point $(x_{k,n},y_{\ell,n})\in D_{k\ell}^n$. It is clear that
for all $(x,y)\in D_{k\ell}^n$,
\begin{align}
  |F_{ij}(x_{k,n},y_{\ell,n},f(x,y))
  -F_{ij}(x,y,f(x,y))|  \leq \sqrt{2} C_{ij} \gep_n. \label{eq:Fkl-est}
\end{align}
On the other hand, for all $(x',y'),(x'',y'') \in D_{k\ell}^n$, we have
\begin{align*}
 & F_{ij}(x_{k,n},y_{\ell,n}, f(x',y') - F_{ij}(x_{k,n},y_{\ell,n},f(x'',y'')) \\
  = & S(u_i(x_{k,n}), v_j(y_{\ell,n}))
 \Big( f(x',y') - f(x'',y'') \Big).
\end{align*}
Combining this with \eqref{eq: f-recu-prop}   and  \eqref{eq:Fkl-est}, we have
\begin{align*}
  &|f(u_i(x'),v_j(y')) - f(u_i(x''),v_j(y''))| \\
  \leq & |S(u_i(x_{k,n}), v_j(y_{\ell,n}))|\cdot |f(x',y') - f(x'',y'')| + 2\sqrt{2} C_{ij} \gep_n
\end{align*}
so that
\begin{align*}
  &O(f,(u_i\times v_j)(D_{k\ell}^n)) \leq |S(u_i(x_{k,n}), v_j(y_{\ell,n}))|\cdot O(f,D_{k\ell}^n) + 2\sqrt{2} C_{ij} \gep_n.
\end{align*}
Thus, from  Lemma~\ref{lem:vertScal},
\begin{align*}
&  \sum_{(i,j)\in \Lambda_r(\alpha,\beta)}  O(f,n+1,D_{ij}) \\
   \leq & \sum_{\substack{1\leq k,\ell\leq K^{n}N \\ D_{k\ell}^n\subset \wdt{D}_{\alpha\beta}}} \sum_{(i,j)\in \Lambda_r(\alpha,\beta)}
     |S(u_i(x_{k,n}), v_j(y_{\ell,n}))|\cdot O(f,D_{k\ell}^n) +2\sqrt{2} C_{ij} NK^{n+2} \\
   = & \gamma_{rt}\cdot O(f,n,\widetilde{D}_{\alpha\beta})+2\sqrt{2} C_{ij} NK^{n+2}.
\end{align*}
Similarly, we have
  $$\sum_{(i,j)\in \gL_r(\ga,\gb)} O(f,n+1,D_{ij})\geq \gamma_{rt} O(f,n,\widetilde{D}_{\alpha\beta})-2\sqrt{2} C_{ij} NK^{n+2}$$
so that \eqref{eq: osi-ineq-temp} holds.
\end{proof}



Define $G=(\gamma_{rt})_{m\times m}$, where $\gamma_{rt}=0$ if $\Lambda_r\cap \Lambda_t^\prime=\emptyset$. Given $1\leq r,t\leq m$ and $n\in \bZp$,  a finite sequence $\{i_k\}_{k=0}^n$ in $\{1,2,\ldots,m\}$ is called an \emph{$n$-path} (a path for short) from $t$ to $r$ if $i_0=t$, $i_n=r$, and
  $$\gamma_{i_{k},i_{k-1}}>0, \quad \forall k=1,2,\ldots,n.$$
$r$ and $t$ are called \emph{connected}, denoted by $r\sim t$, if there exist both a path from $r$ to $t$, and a path from $t$ to $r$. We remark that in general, it is possible that there is no path from $r$ to itself.
A subset $V$ of $\{1,2,\ldots,m\}$ is called \emph{connected} if $r\sim t$ for all $r,t\in V$. Furthermore, $V$ is called a \emph{connected component} of $\{1,2,\ldots,m\}$ if $V$ is connected and there is no connected subset $\wdt{V}$ of $\{1,2,\ldots,m\}$ such that $V\subsetneq \wdt{V}$. It is well known that the matrix $G$ is \emph{irreducible} if and only if $\{1,2,\ldots,m\}$ is connected. Please see  \cite{Varga}  for details.

Given a connected component $V=\{r_{1},\dots,r_{t}\}$ of $\{1,2,\ldots,m\}$, where $r_{1}<r_{2}<\dots<r_{t}$, we define a submatrix $G|_V$ of $G$ by
\begin{equation*}
  (G|_V)_{k\ell}=\gamma_{r_{k},r_{\ell}}, \quad 1\leq k,\ell\leq t.
\end{equation*}





\begin{defi}

  Given $(i,j),(k,\ell)\in\Sigma_N\times\Sigma_N$, we call $(k,\ell)$  an \emph{ancestor} of $(i,j)$  if there exists a finite sequence $\{(i_\tau,j_\tau)\}_{\tau=0}^n$ in $\Sigma_N\times\Sigma_N$ such that $(i_0,j_0)=(i,j)$, $(i_n,j_n)=(k,\ell)$, and
  $$D_{i_\tau, j_\tau}\subset D_{i_{\tau-1},j_{\tau-1}}^\prime, \quad \forall \tau=1,2,\ldots,n.$$
  Let  $A_{0}(i,j)$ be the  collection of  all ancestors of $(i,j)$ and $A(i,j)=\{(i,j)\}\bigcup A_{0}(i,j)$.
  We call $(i,j)$ \emph{degenerate} if for all $(k,\ell)\in A(i,j)$, we have either
  \begin{equation}\label{eq:s-degenerate}
    s_{k-1,\ell-1}=s_{k-1,\ell}=s_{k,\ell-1}=s_{k\ell}=0, \quad \textrm{or}
  \end{equation}
  \begin{equation}\label{eq:degenerate}
    z_{pq}=g_{k\ell}(x_p,y_q), \quad\forall p,q\in \Sigma_{N,0} \textrm{ with $(x_p,y_q)\in D_{k\ell}^\prime$.}
  \end{equation}
  Given $1\leq r\leq m$, we call $r$ is degenerate if $(i,j)$ is degenerate for all $(i,j)\in \Lambda_r$.
  A connected component $V$ of $\{1,2,\ldots,m\}$ is called \emph{degenerate} if $r$ is degenerate for all $r\in V$. Otherwise, $V$ is called non-degenerate.
\end{defi}

By definition, a connected component $V$ of $\{1,2,\ldots,m\}$ is non-degenerate if there exists $(i,j)\in \cup\{\Lambda_r:\, r\in V\}$ such that $(i,j)$ is non-degenerate.

\begin{rem}\label{rem:colliner}
It is clear that \eqref{eq:s-degenerate} holds if and only if $S|_{D_{k\ell}}=0$.
Since $g_{k\ell}$ is bilinear, we can see that \eqref{eq:degenerate} holds if and only if the following two conditions hold:
\begin{enumerate}
  \item for each $p\in \Sigma_{N,0}$ with $x_p\in I'_k$, points in  $\{(x_p,y_q,z_{pq}):\, q\in \Sigma_{N,0}, y_q\in J'_\ell\}$ are collinear; and
  \item for each $q\in \Sigma_{N,0}$ with $y_q\in J'_\ell$, points in  $\{(x_p,y_q,z_{pq}):\, p\in \Sigma_{N,0}, x_p\in I'_k\}$ are collinear.
\end{enumerate}
\end{rem}

\begin{lem}\label{lem:degenerate}
  Given $(i,j)\in\Sigma_N\times \Sigma_N$, if $(i,j)$ is degenerate, then
  \begin{enumerate}
    \item for all $(k, \ell)\in A(i,j)$, we have $f(x,y)=h(x,y)$, $(x,y)\in D_{k\ell}$,
    \item there exists a constant $C>0$, such that $O(f,n,D_{ij})\leq C K^n$ for all $n\in \bZ^+$,
    \item $\dim_B \Gamma f|_{D_{ij}}=2$.
  \end{enumerate}
\end{lem}
\begin{proof}
  (1).~~ Let $T$ and $C^*([0,1]^2)$ be the same as defined in the proof of Theorem~\ref{theo: general-RFIF}. Let $C_{ij}^{*}([0,1]^2)$ be the collection of all continuous functions $\varphi\in C^{*}([0,1]^2)$ satisfying
  $$\varphi(x,y)=h(x,y), \qquad (x,y)\in D_{k\ell},$$
  for all $(k,\ell)\in A(i,j)$.

  Given $\varphi\in C_{ij}^{*}([0,1]^2)$ and $(k,\ell)\in A(i,j)$, if $S|_{D_{k\ell}}\not=0$, then \eqref{eq:degenerate} holds. In this case,  for all $(\ga,\gb)\in \Sigma_N\times \Sigma_N$ with $D_{\ga\gb}\subset D_{k\ell}^\prime$, we have $(\ga,\gb)\in A(i,j)$ so that
    $$\varphi(x,y)=h(x,y), \qquad (x,y)\in D_{\ga\gb}.$$
  Combining this with  \eqref{eq:degenerate}, we know that $\varphi$ is bilinear on $D_{k\ell}^\prime$. Thus,
    $$\varphi(x,y)=g_{k\ell}(x,y), \qquad (x,y)\in D_{k\ell}^\prime.$$
  Hence, using \eqref{eq:T-def} and \eqref{eq: Fij-def}, we have
  \begin{align*}
     T\varphi(x,y)&=F_{k\ell}\big( u_k^{-1}(x), v_\ell^{-1}(y), \varphi(u_k^{-1}(x),v_\ell^{-1}(y)) \big) \\
                  &=S(x,y)\Big( \varphi\big(u_k^{-1}(x),v_\ell^{-1}(y) \big) - g_{k\ell} \big(u_k^{-1}(x),v_\ell^{-1}(y) \big) \Big)+h(x,y) \\
                  &=h(x,y)
  \end{align*}
  for all $(x,y)\in D_{k\ell}$. In the case that $S|_{D_{k\ell}}=0$, it is clear that we still have $T\varphi(x,y)=h(x,y)$ on $D_{k\ell}$ by using \eqref{eq:T-def} and \eqref{eq: Fij-def}. Thus $T$ is a map from $C_{ij}^{*}([0,1]^2)$ to itself. Notice that $C_{ij}^{*}([0,1]^2)$ is complete since it is closed in $C^{*}([0,1]^2)$. Hence $f\in C_{ij}^{*}([0,1]^2)$.

  (2) and (3) directly follow from (1).
\end{proof}

Given a matrix $X=(X_{ij})_{n\times n}$, we say $X$ is non-negative if $X_{ij}\geq0$ for all $i$ and $j$. $X$ is called strictly positive if $X_{ij}>0$ for all $i$ and $j$.
The following lemma is well known. Please see \cite{Varga} for details.
\begin{lem}[Perron-Frobenius Theorem]
Let $X=(X_{ij})_{n\times n}$ be an irreducible non-negative matrix. Then
   \begin{enumerate}
   \item $\rho(X)$, the spectral radius of $X$, is an eigenvalue of $X$ and has strictly positive eigenvector.
   \item $\rho(X)$ increases if any element of $X$ increases.
     \end{enumerate}
\end{lem}


  Given three points $(x^{(k)},y)$, $k=1,2,3$ in $[0,1]^2$ with $(x^{(2)}-x^{(1)})(x^{(3)}-x^{(2)})>0$, we denote
  \begin{equation*}
    d_f(x^{(1)},x^{(2)},x^{(3)};y)=|f(x^{(2)},y)-\big(\lambda f(x^{(1)},y) + (1-\lambda) f(x^{(3)},y) \big)|,
  \end{equation*}
  where $\lambda=(x^{(3)}-x^{(2)})/(x^{(3)}-x^{(1)})$. By definition, $(x^{(k)},y,f(x^{(k)},y)),k=1,2,3$ are collinear if and only if $d_f(x^{(1)},x^{(2)},x^{(3)};y)=0$. Furthermore, if all of $(x^{(k)},y)$, $k=1,2,3$ lie in a subset $E$ of $[0,1]^2$, then $O(f,E)\geq d_f(x^{(1)},x^{(2)},x^{(3)};y)$.

The following result is the key lemma to obtain the exact box dimension of bilinear RFISs.

\begin{lem}\label{lem:Osi-global-ineq}
Let $V$ be a non-degenerate connected component of $\{1,2,\ldots,m\}$. If $\rho(G|_{V})>K$, then
  for all $r\in V$,
  \begin{align}\label{eq:lem-4.9}
    \lim_{n\rightarrow\infty}\frac{O(f,n,B_{r})}{K^{n}}=\infty.
  \end{align}
\end{lem}
\begin{proof}

  Denote by $n_V$ the cardinality of $V$, and assume that $V=\{r_1,\ldots,r_{n_V}\}$. Firstly, we show that following claim holds: \emph{for all $1\leq k\leq n_V$,
  there exist $(i^{(k)},j^{(k)})\in \Lambda_{r_{k}}$ and  three points $(x_{k,\tau},y^{(k)})$, $\tau=1,2,3$ in $D_{i^{(k)},j^{(k)}}$, such that
  \begin{equation*}
    \delta_k:=d_f(x_{k,1},x_{k,2},x_{k,3};y^{(k)})>0.
  \end{equation*}}

  Since $V$ is non-degenerate, there exist $t^*\in V$ and $(i^*,j^*)\in \gL_{t^*}$ such that $(i^*,j^*)$ is non-degenerate. Thus there exists $(k_{0},\ell_{0})\in A(i^*,j^*)$ such that both \eqref{eq:s-degenerate} and \eqref{eq:degenerate} do not hold. By Remark~\ref{rem:colliner}, we can assume without loss of generality that
  there exist three interpolation points $\{(x_{p_i},y_q,z_{p_i,q})\}_{1\leq i\leq 3}$ in $D'_{k_0\ell_0}$ with $p_1<p_2<p_3$, such that they are not collinear. That is, if we denote $\lambda_0=(x_{p_3}-x_{p_2})/(x_{p_3}-x_{p_1})$, then
  \begin{align*}
    \delta_0 & :=d_f(x_{p_1},x_{p_2},x_{p_3};y_q)  =|z_{p_2,q} -(\lambda_0 z_{{p_1},q}+(1-\lambda_0)z_{p_3,q})| >0.
  \end{align*}

  Let $r_0$ and $t_0$ be elements in $\{1,2,\ldots,m\}$ satisfying $(k_0,\ell_0)\in \gL_{r_0}\cap \gL_{t_0}'$.
  Let $\ga$ and $\gb$ be elements in $\{0,1,\ldots,N-K\}$ satisfying $[x_\ga,x_{\ga+K}]\times [y_\gb,y_{\gb+K}]=D_{k_{0}\ell_{0}}'$.
  Notice that $g_{ij}$ are the same for all $(i,j)\in \gL_{r_0}(\ga,\gb)$, since they are the same bilinear function passing through the four points $(x_\ga,y_\gb,z_{\ga,\gb})$, $(x_\ga,y_{\gb+K},z_{\ga,\gb+K})$, $(x_{\ga+K},y_\gb,z_{\ga+K,\gb})$ and $(x_{\ga+K},y_{\gb+K},z_{\ga+K,\gb+K})$. We denote it by $\widetilde{g}_{\alpha\beta}$.

  For $(i,j)\in\gL_{r_0}(\ga,\gb)$, we define $\theta_{ij}=1$ if $S$ is nonnegative on $D_{ij}$, and define $\theta_{ij}=-1$ otherwise. Then $|S(x,y)|=\theta_{ij}S(x,y)$ for all $(x,y)\in D_{ij}$. Combining this with \eqref{eq: f-recu-prop} and \eqref{eq: Fij-def},
  \begin{align*}
    \theta_{ij}f(u_i(x),v_j(y))=|S(u_i(x),v_j(y))|\big(f(x,y)-g_{ij}(x,y)\big) +\theta_{ij}h(u_i(x),v_j(y))
  \end{align*}
  for $(x,y)\in D_{ij}'=[x_\ga,x_{\ga+K}]\times [y_\gb,y_{\gb+K}]$. By  Lemma~\ref{lem:vertScal},
  \begin{align*}
    &\sum_{(i,j)\in \gL_{r_0}(\ga,\gb)}\theta_{ij} f(u_i(x),v_j(y)) \\
    = &\gamma_{r_0 t_0} \big(f(x,y) - \widetilde{g}_{\alpha\beta}(x,y)\big)
      +\sum_{(i,j)\in \gL_{r_0}(\ga,\gb)} \theta_{ij}h(u_i(x),v_j(y))
  \end{align*}
  for all $(x,y)\in [x_\ga,x_{\ga+K}]\times [y_\gb,y_{\gb+K}]$.
  On the other hand, since both $\widetilde{g}_{\ga\gb}$ and $h|_{D_{ij}}$, $(i,j)\in \gL_{r_0}(\ga,\gb)$ are bilinear, we have
  \begin{align*}
    &\wdt{g}_{\ga\gb}(x_{p_2},y_q)=\lambda_0 \wdt{g}_{\ga\gb}(x_{p_1},y_q)+(1-\lambda_0) \wdt{g}_{\ga\gb}(x_{p_3},y_q), \quad \textrm{and}\\
    &h(u_i(x_{p_2}),v_j(y_q))=\lambda_0 h(u_i(x_{p_1}),v_j(y_q))+(1-\lambda_0) h(u_i(x_{p_3}),v_j(y_q))
  \end{align*}
  for all $(i,j)\in \gL_{r_0}(\ga,\gb)$. Thus
  \begin{align*}
    &\sum_{(i,j)\in \gL_{r_0}(\ga,\gb)} d_f(u_i(x_{p_1}),u_i(x_{p_2}),u_i(x_{p_3});v_j(y_q)) \\
    \geq & \, \Big|\sum_{(i,j)\in \gL_{r_0}(\ga,\gb)}\theta_{ij}f(u_i(x_{p_2}),v_j(y_q))-\lambda_0\sum_{(i,j)\in \gL_{r_0}(\ga,\gb)}\theta_{ij}f(u_i(x_{p_1}),v_j(y_q))\\
    &\hspace{2em}-(1-\lambda_0)\sum_{(i,j)\in \gL_{r_0}(\ga,\gb)}\theta_{ij}f(u_i(x_{p_3}),v_j(y_q))\Big|\\
    = &\, \gamma_{r_0 t_0}\big|f(x_{p_2},y_q)-\big(\lambda_0 f(x_{p_1},y_q)+(1-\lambda_0)f(x_{p_3},y_q)\big)\big|
                 =\gamma_{r_0t_0}\delta_0.
  \end{align*}
  As a result, there exists $(i,j)\in \gL_{r_0}(\ga,\gb)$ such that
  \begin{align*}
    d_f(u_i(x_{p_1}),u_i(x_{p_2}),u_i(x_{p_3});v_j(y_q))>0.
  \end{align*}
  We remark that $(u_i(x_{p_\tau}),v_j(y_q))$, $\tau=1,2,3$ are three points in $D_{ij}$.

  From $(i^*,j^*)\in \gL_{t^*}$ and $(k_{0},\ell_{0})\in A(i^*,j^*)\cap \gL_{r_0}\cap \gL_{t_0}'$, there is a path from $t_0$ to $t^*$. Since $V$ is connected, we know that for all $1\leq k\leq n_V$, there is a path  from $t^*$ to $r_{k}$ so that there is a path from $t_0$ to $r_{k}$.
  Similarly as above, we can prove by induction that the claim holds.

  Now we will show \eqref{eq:lem-4.9} holds if $\rho(G|_{V})>K$.

  Let $\xi=(\xi_{1},\xi_{2},\dots,\xi_{n_V})^T$ be a strictly positive eigenvector of $G|_{V}$ with eigenvalue $\rho(G|_{V})$ such that $\xi_{k}\leq \delta_{k}$ for all $1\leq k\leq n_V$.

  Given $1\leq k,\ell \leq n_V$ and $n\in \bZ^+$, we denote by $\PVn(r_k,r_\ell)$ the set of all $n$-paths in $V$ from $r_k$ to $r_\ell$. Given $\overrightarrow{r} =\{r(0),r(1),\ldots,r(n)\}\in\PVn(r_k,r_\ell)$, we denote by $\mathcal{Q}(\overrightarrow{r})$ the set of all elements $(i_0i_1\cdots i_n,j_0 j_1\cdots j_n)$ in $\Sigma_N^{n+1}\times \Sigma_N^{n+1}$ satisfying $(i_0,j_0)=(i^{(k)},j^{(k)})$,   $D_{i_{\tau-1},j_{\tau-1}}\subset D_{i_{\tau},j_{\tau}}^\prime$ and $(i_{\tau},j_{\tau})\in\gL_{r(\tau)}$ for all $1\leq \tau\leq n$.

  Given $(\mathbf{i},\mathbf{j})=(i_0i_1\cdots i_n,j_0j_1\cdots j_n)\in \mathcal{Q}(\overrightarrow{r})$ and $(x,y)\in D_{i_0j_0}$, we define
   \begin{align*}
    u_{\mathbf{i}}(x)=u_{i_{n}}\circ u_{i_{n-1}}\circ \cdots \circ u_{i_{1}}(x), \qquad v_{\mathbf{j}}(y)=v_{j_{n}}\circ v_{j_{n-1}}\circ \cdots \circ v_{j_{1}}(y).
   \end{align*}

   Given $1\leq k,\ell\leq n_V$ and $n\geq 1$, for each $\overrightarrow{r} =\{r(0),r(1),\cdots, r(n)\}\in\PVn(r_k,r_\ell)$,
  similarly as above, we have
  \begin{align*}
    \sum_{(\bf i,\bf j)\in\mathcal{Q}(\overrightarrow{r})}  d_f(u_{\bf i}(x_{k,1}),u_{\bf i}(x_{k,2}),u_{\bf i}(x_{k,3}); v_{\bf j}(y^{(k)}))
    =\delta_{k}\prod_{t=1}^n \gamma_{r(t),r(t-1)}
  \end{align*}
  so that
  \begin{align*}
      &\sum_{\overrightarrow{r}\in\mathcal{P}_V^{(n)}(r_{k},r_{\ell})}
    \sum_{(\bfi,\bfj)\in \mathcal{Q}(\overrightarrow{r})} O \big(f, (u_{\bfi}\times v_{\bfj})(D_{i^{(k)},j^{(k)}}) \big) \\
    \geq \, &
    \delta_{k}\sum_{\{r(0), \ldots, r(n)\}\in\mathcal{P}_V^{(n)}(r_{k},r_{\ell})}
    \prod_{t=1}^n \gamma_{r(t),r(t-1)}\\
    = \, &\delta_{k}\sum_{\substack{r(1), \ldots, r(n-1)\in V \\ r(0)=r_k,r(n)=r_\ell}}
    \prod_{t=1}^n \gamma_{r(t),r(t-1)}=\delta_k\big((G|_V)^n\big)_{\ell k}.
  \end{align*}
  As a result,
  \begin{align*}
  O(f,n,B_{r_{\ell}})
  &\geq
  \sum_{1\leq k\leq n_V}
  \sum_{\overrightarrow{r}\in\mathcal{P}_V^{(n)}(r_{k},r_{\ell})}
  \sum_{(\bfi,\bfj)\in \mathcal{Q}(\overrightarrow{r})} O \big(f, (u_{\bfi}\times v_{\bfj})(D_{i^{(k)},j^{(k)}}) \big) \\
  &\geq
  \sum_{1\leq k\leq n_V}\big((G|_V)^n\big)_{\ell k} \delta_k.
  \end{align*}
  Thus
  \begin{align*}
    \big(O(f,n,B_{r_1}), \ldots, O(f,n,B_{r_{n_V}}) \big)^T & \geq ( G|_V)^n (\delta_1,\ldots,\delta_{n_V})^T \\
    &\geq (G|_V)^n (\xi_1,\ldots,\xi_{n_V})^T \\
    & \geq (\rho(G|_V))^n (\xi_1,\ldots,\xi_{n_V})^T.
  \end{align*}
  Hence for all $1\leq k\leq n_V$, we have
  \begin{align*}
    O(f,n,B_{r_k})\geq (\rho(G|_{V}))^{n}\xi_k.
  \end{align*}
  The lemma follows from $\rho (G|_{V})>K$.
\end{proof}


For each $1\leq r\leq m$, we define
 \begin{align*}
   A^*(r)=\{t\in \{1,\ldots,m\}:\, \mbox{$t\not\sim r$ and there exists a path from $t$ to $r$} \}.
 \end{align*}
Then we recurrently define $P(r)$, $r=1,2,\ldots,m$ as follows: $P(r)=1$ if $A^*(r)=\emptyset$, and $P(r)=1+\max\{P(t):\, t\in A^*(r)\}$ if $A^*(r)\not=\emptyset$. We call $P(r)$ the \emph{position} of $r$.

Now we can use Lemmas~\ref{lem:Osi-local-ineq-1} and \ref{lem:Osi-global-ineq} to obtain the exact box dimension of bilinear RFISs under certain constraints. The spirit of the proof follows from \cite{BEH89,KRZ18,RSY00}.
\begin{theo}\label{thm: box-dim}
Let $f$ be the bilinear RFIF determined in the section~\ref{sec:bilinear} with conditions \eqref{eq:hom-cond-1}, \eqref{eq:hom-cond-2} and \eqref{eq:hom-cond-3}.   Assume that the vertical scaling factors $\{s_{ij}:\, i,j\in \Sigma_{N,0}\}$ are steady and have uniform sums $\{\gamma_{rt}:\, \Lambda_r\cap \Lambda_t^\prime\not=\emptyset\}$ under a compatible partition $\SB=\{B_r\}_{r=1}^m$.  Let  $\{V_1,V_2,\ldots,V_{n_c^*}\}$ be the set of all non-degenerate connected components of $\{1,2,\ldots,m\}$. Define $d_i=\log \rho(G|_{V_i})/ \log K$ for $ 1\leq i \leq n_c^*$ and let $d_*=\max\{d_{1},\dots,d_{n_c^*},1\}$. Then $\dim_B (\Gamma f)=1+d_*$.
\end{theo}

\begin{proof}
Denote $\rho_0=K^{d_*}$. Firstly, we will prove $\underline{\dim}_{B} (\Gamma f)\geq 1+d_*$. It is clear that $\underline{\dim}_{B} (\Gamma f)\geq2$ since $f$ is continuous. Thus we only need to consider the case $d_*>1$. Let $i_0$ be an element in $\{1,\ldots,n_c^*\}$ satisfying $d_{i_0}=d_*$. Assume that $V_{i_0}=\{r_{1},\dots,r_{q}\}$ for some $1\leq q \leq m$, where $r_1<\cdots<r_q$. Then $\rho_{0}=\rho(G|_{V_{i_0}})$. Let $\xi=(\xi_{1},\ldots,\xi_{q})^T$ be a strictly positive eigenvector of $G|_{V_{i_0}}$ with eigenvalue $\rho_{0}$. Then
\begin{equation*}
  \sum_{\ell=1}^q \gamma_{r_k,r_\ell} \xi_\ell=\rho_0\xi_k
\end{equation*}
for all $1\leq k,\ell\leq q$. From Lemma~\ref{lem:Osi-local-ineq-1}, there exists a  constant $C>0$ such that
   \begin{equation*}
   O(f,n+1,B_{r_{k}})\geq \sum_{t=1}^m \gamma_{r_{k},t} O(f,n,B_{t})-CK^{n} \label{eq:N-rec-relly}
  \end{equation*}
   for all $1\leq k\leq q$ and $n\in \mathbb{Z}^+$.  Choose a constant $C_{1}>0$ such that $C_{1}\xi_{k}\geq C$ for all $k$. Then for $1\leq k\leq q$ and $n\in \mathbb{Z}^+$, we have
   \begin{align}\label{ineq:Osi-Brk-rec}
      O(f,n+1,B_{r_{k}})\geq \sum_{\ell=1}^q \gamma_{r_{k},r_{\ell}} O(f,n,B_{r_{\ell}})-C_{1}\xi_{k}K^{n}.
  \end{align}


It follows from $d_*>1$ that $\rho_0=K^{d_*}>K$. Thus, from Lemma~\ref{lem:Osi-global-ineq}, we can pick $n_0\in \mathbb{Z}^+$ large enough and $C_2>0$ small enough such that
 \begin{equation*}
   O(f,n_0,B_{r_{k}})\geq C_{2} \rho_0^{n_0 }\xi_{k}+\frac{C_{1}K^{n_0}}{\rho_{0}-K}\xi_{k}
 \end{equation*}
 for $1\leq k\leq q$. Combining this with \eqref{ineq:Osi-Brk-rec}, we know that for $1\leq k\leq q$,
\begin{align*}
  O(f,n_0+1,B_{r_{k}})
  &\geq C_{2}\rho_0^{n_0}\sum_{\ell=1}^q \gamma_{r_{k},r_{\ell}}\xi_{\ell} + \frac{C_{1}K^{n_0}}{\rho_{0}-K} \sum_{\ell=1}^q \gamma_{r_{k},r_{\ell}}\xi_{\ell}-C_{1}\xi_{k}K^{n_0}\\
  & =C_{2}\rho_0^{n_0+1}\xi_{k} + \frac{C_1 K^{n_0+1}}{\rho_0-K} \xi_k.
\end{align*}
By induction, we can obtain that for all $1\leq k\leq q$ and $n\geq n_0$,
\begin{equation*}
  O(f,n,B_{r_{k}})\geq C_{2} \rho_0^{n }\xi_{k}+\frac{C_{1}K^{n}}{\rho_{0}-K}\xi_{k}
    \geq  C_{2} K^{n d_*}\xi_{k}.
\end{equation*}
Notice that $O(f,n)\geq O(f,n,B_{r_1})$ for all $n\in \bZp$. Thus using Lemma~\ref{lem:box-dim}, we have $\underline{\dim}_{B}(\Gamma f)\geq 1+d_*$.

Now we will show that the following claim holds:\, \emph{for all $1\leq r\leq m$ and all $\delta>0$, there exists $\wdt{C}>0$ such that
\begin{equation}\label{ineq:upperBoxD-Est}
  O(f,n,B_r)\leq \wdt{C}(\rho_0+\delta)^n
\end{equation}
for all $n\in \bZ^+$.}

We will prove this by using induction on $P(r)$. Denote $P_{\max}=\max\{P(r): \, 1\leq r\leq m\}$.


In the case that $P(r)=1$, we have $A^*(r)=\emptyset$. If $r$ is  degenerate, then from Lemma~\ref{lem:degenerate} (2) and $\rho_0=K^{d_*}\geq K$, we know that the claim holds. Thus we can assume that $r$ is  non-degenerate. Combining this with $P(r)=1$, there exists a non-degenerate connected component $V_{i}$ such that $r\in V_{i}$.   Assume that $V_{i}=\{r_{1},\dots,r_q\}$ for some $1\leq q\leq m$, where $r_1<\cdots<r_q$. Denote $\rho_{i}=\rho(G|_{V_{i}})$. Let $\eta=(\eta_{1},\ldots,\eta_{q})^T$ be a strictly positive eigenvector of $G|_{V_{i}}$ with eigenvalue $\rho_{i}$. It follows from $A^*(r)=\emptyset$ that $\gamma_{r_k,t}=0$ for all $1\leq k\leq q$ and $t\not\in V_i$. Thus, from Lemma~\ref{lem:Osi-local-ineq-1}, there exists a  constant $C>0$ such that
   \begin{equation*}
   O(f,n+1,B_{r_{k}})\leq \sum_{\ell=1}^{q}\gamma_{r_{k},r_\ell} O(f,n,B_{r_\ell})+CK^{n} \label{eq:N-rec-relly}
  \end{equation*}
for all $1\leq k\leq q$  and $n\in \mathbb{Z}^+$.  Choose a constant $C_{1}>0$ such that $C_{1}\eta_{k}\geq C$ for all $1\leq k\leq q$.
Then for all $1\leq k\leq q$  and $n\in \mathbb{Z}^+$,
   \begin{align}\label{eq:N-rec-rellysut}
     O(f,n+1,B_{r_{k}})&\leq \sum_{\ell=1}^{q}\gamma_{r_{k},r_{\ell}} O(f,n,B_{r_{\ell}})+C_{1} \rho_0^{n} \eta_{k} .
  \end{align}

Arbitrarily pick $\delta>0$. Let $C_{2}$ be a positive constant such that
\begin{equation*}
O(f,1,B_{r_{k}})\leq C_{2} \rho_i \eta_k + C_1 \delta^{-1} (\rho_0+\delta)\eta_{k}
\end{equation*}
 for all $1\leq k\leq q$. By induction and using \eqref{eq:N-rec-rellysut}, we have
\begin{equation}\label{eq:N-rec-rellsu}
  O(f,n,B_{r_{k}}) \leq C_{2} \rho_i^n \eta_k + C_1 \delta^{-1} (\rho_0+\delta)^n \eta_{k}
\end{equation}
for all $1\leq k\leq q$  and $n\in \mathbb{Z}^+$, where we use
\begin{align*}
  &\sum_{\ell=1}^q \gamma_{r_k,r_\ell} \Big( C_2 \rho_i^n \eta_\ell + C_1 \delta^{-1} (\rho_0+\delta)^n\eta_\ell \Big) + C_1 \rho_0^n \eta_k \\
  \leq \, & C_2 \rho_i^{n+1} \eta_k + C_1 \delta^{-1} (\rho_0+\delta)^n \rho_i \eta_k + C_1(\rho_0+\delta)^n \eta_k \\
  \leq \, & C_2 \rho_i^{n+1} \eta_k + C_1 \delta^{-1} (\rho_0+\delta)^{n+1} \eta_{k}.
\end{align*}
It follows that the claim holds if $P(r)=1$.

Assume that the claim holds for all $1\leq r\leq m$ satisfying $P(r)\leq P$, where $1\leq P<P_{\max}$. Let $r$ be an element in $\{1,2,\ldots,m\}$ satisfying $P(r)= P+1$. It is clear that \eqref{ineq:upperBoxD-Est} holds if $r$ is degenerate. Thus we can assume that $r$ is non-degenerate. In the case that $r$ does not belong to any connected component, we have $P(t)\leq P$ for any $1\leq t\leq m$ with $\gamma_{rt}>0$.
Combining this with Lemma~\ref{lem:Osi-local-ineq-1}, there exists a constant $C>0$, such that
\begin{align*}
  O(f,n+1,B_r) \leq \Big( \sum_{t=1}^m \gamma_{rt} \Big) \sum_{t:\, P(t)\leq P} O(f,n,B_t) + C K^n, \qquad \forall n\geq 1.
\end{align*}
By inductive assumption, we can see that the claim holds in this case.

Now we consider the case that $r$  belongs to a connected component $V_{i}=\{r_{1},\dots,r_{q}\}$, where $r_1<\cdots<r_q$. Similarly as above, let $\eta=(\eta_{1},\ldots,\eta_{q})^T$ be a strictly positive eigenvector of $G|_{V_{i}}$ with eigenvalue $\rho_i=\rho(G|_{V_{i}})$.  From Lemma~\ref{lem:Osi-local-ineq-1}, there exists a  constant $C>0$ such that
   \begin{equation*}
   O(f,n+1,B_{r_{k}})\leq \sum_{\ell=1}^q \gamma_{r_{k},r_\ell} O(f,n,B_{r_\ell})+\sum_{t: P(t)\leq P }\gamma_{r_{k},t} O(f,n,B_{t})+CK^{n} \label{eq:N-rec-rellys}
  \end{equation*}
   for all $1\leq k\leq q$  and $n\in \mathbb{Z}^+$. Thus, given $\delta>0$, by using inductive assumption, there exists a  constant $C_{1}>0$ such that
   \begin{align}\label{eq:N-rec-rellysses}
  O(f,n+1,B_{r_{k}})\leq \sum_{\ell=1}^{q}\gamma_{r_{k},r_{\ell}} O(f,n,B_{r_{\ell}})+C_{1}\eta_{k}(\rho_{0}+\delta)^{n}
  \end{align}
  for all $1\leq k\leq q$  and $n\in \mathbb{Z}^+$.
  Similarly as above, there exists $C_{2}>0$ such that for all $1\leq k\leq \ell$  and $n\in \mathbb{Z}^+$,
\begin{equation*}
  O(f,n,B_{r_{k}}) \leq C_{2}\rho_{i}^{n}\eta_k + C_{1}\delta^{-1}(\rho_{0}+\delta)^{n}\eta_{k}.
\end{equation*}
Hence the claim holds for $P(r)=P+1$.

By induction, the claim holds for all $1\leq r\leq m$. Combining this with Lemma~\ref{lem:box-dim}, we can see from the arbitrariness of $\delta$ that $\overline{\dim}_B \Gamma f|_{B_{r}}\leq 1+d_{*}$ for all $1\leq r\leq m$. Thus $\overline{\dim}_B (\Gamma f)\leq 1+d_{*}$.
As a result, $\dim_B (\Gamma f)= 1+d_{*}$.
\end{proof}

We can  obtain the following corollary immediately.
\begin{cor}
Let $f$ be the bilinear RFIF determined in the section~\ref{sec:bilinear} with conditions \eqref{eq:hom-cond-1}, \eqref{eq:hom-cond-2} and \eqref{eq:hom-cond-3}.   Assume that the vertical scaling factors $\{s_{ij}:\, i,j\in \Sigma_{N,0}\}$ are steady and have uniform sums $\{\gamma_{rt}:\, \Lambda_r\cap \Lambda_t^\prime\not=\emptyset\}$ under a compatible partition $\SB=\{B_i\}_{i=1}^m$. Then, in the case that $G$ is irreducible, we have
\begin{enumerate}
  \item $\dim_B (\Gamma f)=1+\log \rho(G)/\log K$ if $\{1,2,\ldots,m\}$ is non-degenerate and $\rho(G)>K$, or
  \item $\dim_B (\Gamma f)=2$ otherwise.
\end{enumerate}
\end{cor}
\begin{rem}
  In the case that $K=N$, it is clear that $\mathcal{B}=\{B_1\}=\{[x_0,x_N]\times[y_0,y_N]\}$ is the compatible partition. From Theorem~\ref{thm: box-dim}, we can obtain the box dimension of bilinear FISs under certain constraints, which is the main result in \cite{KRZ18}.
\end{rem}

\subsection{An example}

\begin{figure}[htbp]
    \center
  \includegraphics[height=5cm]{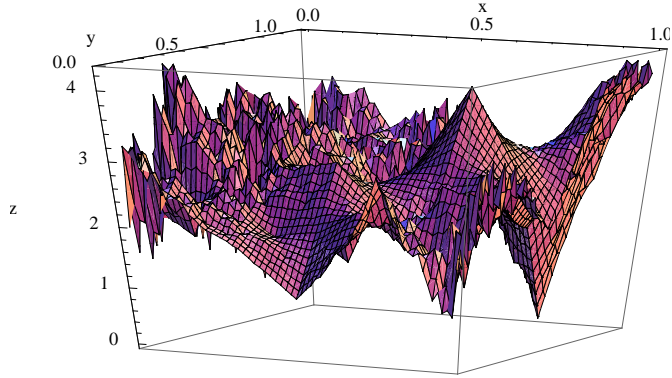}
  \caption{Bilinear RFIS in Example~\ref{ex-fig}}
  \label{figure: FIS-2019}
\end{figure}

\begin{exam}\label{ex-fig}
Let $N=4$, $K=2$, and $x_{0}^\prime=x_{2}^\prime=x_{4}^\prime=0$, $x_{1}^\prime=x_{3}^\prime=\frac{1}{2}$,  $y_{0}^\prime=y_{2}^\prime=y_{4}^\prime=\frac{1}{2}$, $y_{1}^\prime=1$, $y_{3}^\prime=0$. By definition, we have
\begin{align*}
  &u_1(x_0)=x_0, && u_1(x_2)=u_2(x_2)=x_1, && u_2(x_0)=u_3(x_0)=x_2, \ldots, \\
  &v_1(y_2)=y_0, && v_1(y_4)=v_2(y_4)=y_1, && v_2(y_2)=v_3(y_2)=y_2, \ldots.
\end{align*}
Let $B_{1}=[0,\frac{1}{2}]\times[0,\frac{1}{2}]$, $B_{2}=[0,\frac{1}{2}]\times[\frac{1}{2},1]$ and $B_{3}=[\frac{1}{2},1]\times[0,1]$. It is easy to check that $\mathcal{B}=\{B_1,B_2,B_3\}$ is a compatible partition with respect to $\{D_{ij},D_{ij}^\prime;1\leq i,j\leq 4\}$.
Furthermore, $\Lambda_r\cap \Lambda_t^\prime\not=\emptyset$ if and only if $(r,t)\in \{(1,2),(2,1), (3,1), (3,2)\}$.

Let $(r,t)=(1,2)$. Then $\Lambda_1(0,2)=\{(1,1),(1,2),(2,1),(2,2)\}$ so that
\begin{align*}
  \sum_{(i,j)\in \Lambda_1(0,2)} |S(u_i(x_0),v_j(y_2))|
  & = |S(x_0,y_0)|+|S(x_0,y_2)|+|S(x_2,y_0)|+|S(x_2,y_2)|\\
  &=|s_{00}|+|s_{02}|+|s_{20}|+|s_{22}|.
\end{align*}
Similarly, we have
\begin{align*}
  &\sum_{(i,j)\in \Lambda_1(0,2)} |S(u_i(x_0),v_j(y_4))|=2(|s_{01}|+|s_{21}|),\\
  &\sum_{(i,j)\in \Lambda_1(0,2)} |S(u_i(x_2),v_j(y_2))|=2(|s_{10}|+|s_{12}|),\\
  &\sum_{(i,j)\in \Lambda_1(0,2)} |S(u_i(x_2),v_j(y_4))|=4|s_{11}|.
\end{align*}

Assume that vertical scaling factors $\{s_{ij}:\, 0\leq i,j\leq 4\}$ have uniform sums under the compatible partition $\{B_1,B_2,B_3\}$. Then we must have
\begin{equation*}
  \gamma_{12}=|s_{00}|+|s_{02}|+|s_{20}|+|s_{22}|=2(|s_{01}|+|s_{21}|)=2(|s_{10}|+|s_{12}|)=4|s_{11}|.
\end{equation*}
Similarly,
\begin{align*}
  \gamma_{21}=|s_{02}|+|s_{04}|+|s_{22}|+|s_{24}|=2(|s_{03}|+|s_{23}|)=2(|s_{12}|+|s_{14}|)=4|s_{13}|,\\
  \gamma_{31}=|s_{22}|+|s_{24}|+|s_{42}|+|s_{44}|=2(|s_{23}|+|s_{43}|)=2(|s_{32}|+|s_{34}|)=4|s_{33}|,\\
  \gamma_{32}=|s_{20}|+|s_{22}|+|s_{40}|+|s_{42}|=2(|s_{21}|+|s_{41}|)=2(|s_{30}|+|s_{32}|)=4|s_{31}|.
\end{align*}

Let
\begin{equation*}
(s_{ij})_{0\leq i,j\leq 4}=
\left(
\begin{array}{ccccc}
0.85 & 0.9 & 0.95 & 0.9 & 0.9\\
0.2  & 0.45 & 0.8 & 0.7 & 0.6\\
0    &  0   &  0  & 0.5 & 0.95\\
-0.4 & -0.2 &  0  & 0.3 & 0.6\\
-0.8 & -0.4 &  0  & 0.1 & 0.25
\end{array}
\right).
\end{equation*}
By above discussion, we can check that vertical scaling factors are steady and have uniform sums under $\{B_1,B_2,B_3\}$ with
$\gamma_{12}=1.8$, $\gamma_{21}=2.8$, $\gamma_{31}=1.2$ and $\gamma_{32}=0.8$.

Assume that
\begin{equation*}
(z_{ij})_{0\leq i,j\leq 4}=
\left(
\begin{array}{ccccc}
  2 & 3 & 2 & 1 & 2\\
  2 & 2 & 3 & 1 & 3\\
  1 & 3 & 2 & 3 & 1\\
  3 & 2 & 4 & 2 & 0\\
  2 & 3 & 2 & 4 & 4
\end{array}
\right).
\end{equation*}
The corresponding  bilinear RFIS is shown in Figure~\ref{figure: FIS-2019}.

Furthermore, we can check that $V_1=\{1,2\}$ is the unique non-degenerate connected component of $\{1,2,3\}$ with
\begin{equation*}
G|_{V_1}=\left(
\begin{array}{cc}
  0    &  1.8\\
  2.8  &  0\\
\end{array}
\right).
\end{equation*}
Thus $d_{*}= \log 5.04/(2 \log 2)$ so that $\dim_B (\Gamma f)= 1 + \log 5.04 / (2 \log 2)$.
\end{exam}




\bigskip
\bigskip

\begin{center}
{\noindent\bf Acknowledgments}
\end{center}
\medskip

The authors wish to thank Dr. Qing-Ge Kong and Professor Yang Wang for helpful discussions.

\bibliographystyle{amsplain}

\end{document}